\documentclass[11pt]{amsart}
\pdfoutput=1
\usepackage[margin=1.2in]{geometry}

\usepackage{hyperref}

\usepackage{amsfonts}
\usepackage{amsmath}
\usepackage{amsthm}
\usepackage{amssymb}
\usepackage{float}
\usepackage{enumerate}
\usepackage{graphicx}
\usepackage{mathrsfs}

\usepackage{xcolor}

\usepackage{stmaryrd}
\usepackage{MnSymbol}
\usepackage{caption}

\newtheorem{theor}{Theorem}[section]
\newtheorem{cor}[theor]{Corollary}
\newtheorem{lemma}[theor]{Lemma}
\newtheorem{prop}[theor]{Proposition}

\newtheorem{rmk}[theor]{Remark}
\newtheorem{thmintrod}{Theorem}

\DeclareMathOperator{\arccosh}{arccosh}
\DeclareMathOperator{\arcsinh}{arcsinh}
\DeclareMathOperator{\arctanh}{arctanh}
\DeclareMathOperator{\area}{area}

\DeclareMathOperator{\dist}{dist}

\title[The Maximum Injectivity Radius of Hyperbolic Orbifolds]{The Maximum Injectivity Radius\\ of Hyperbolic Orbifolds}
\date{\today}
\author[F. Fanoni]{Federica Fanoni}
\thanks{Research supported by Swiss National Science Foundation grant number PP00P2\textunderscore 128557}
\address{Department of Mathematics, University of Fribourg, Switzerland}
\email{\tt federica.fanoni@unifr.ch}

\begin{document}
\begin{abstract}
For two-dimensional orientable hyperbolic orbifolds, we show that the radius of a maximal embedded disk is greater or equal to an explicit constant $\rho_T$, with equality if and only if the orbifold is a sphere with three cone points of order 2, 3 and 7.
\end{abstract}
\maketitle

\section{Introduction}

The main  goal of this article is to prove the following result:
\begin{thmintrod}
For every hyperbolic orbifold surface, the maximum injectivity radius is greater or equal to a universal constant $\rho_T$, with equality if and only if the orbifold is a sphere with three cone points of order $2$, $3$ and $7$.
\end{thmintrod}
The original interest in this result comes from the study of the automorphism group of a surface (possibly with cusps), and in particular from the following problem: consider a hyperbolic surface $S$ and the group $\mathrm{Aut}^+(S)$ of orientation preserving isometries of $S$. Take a point $p$ on $S$, not fixed by any $\varphi\in\mathrm{Aut}^+(S)$. It is possible to choose a small radius $r$, depending on $S$ and $p$, such that the open disks $\left\{B_r(\varphi(p))\right\}_{\varphi\in\mathrm{Aut}^+(S)}$ are pairwise disjoint, but can we avoid the dependence on the surface and the point? More precisely, is there a universal constant $\rho$ such that for every hyperbolic surface there is a point $p$ with pairwise disjoint embedded disks $\left\{B_{\rho}(\varphi(p))\right\}_{\varphi\in\mathrm{Aut}^+(S)}$? And which is the maximum $\rho$ with this property? As a consequence of Theorem $1$, we can answer both questions.
\begin{thmintrod}
For every hyperbolic surface $S$ there exists a point $p$ such that $$\left\{B_{\rho_T}(\varphi(p))\right\}_{\varphi\in\mathrm{Aut}^+(S)}$$ are embedded and pairwise disjoint. Moreover, $\rho_T$ is the biggest possible radius if and only if $S$ is a Hurwitz surface.
\end{thmintrod}

Various results about the maximum injectivity radius are known for orientable surfaces with or without boundary. In particular, lower bounds have been given by Yamada \cite{yamada} for surfaces without boundary and by Parlier \cite{parlier} for surfaces with boundary. DeBlois \cite{deblois} has given upper bounds, depending on the signature, for surfaces without boundary, generalizing the result of Bavard \cite{bavard} for closed surfaces. For non-orientable surfaces, the only known result is an upper bound in the closed case, again by Bavard \cite{bavard}.

The structure of the article is the following: after some preliminary results, in section \ref{triangularsurfaces} we study the maximum injectivity radius of triangular surfaces (cone-surfaces of genus zero with three singular points) and we obtain the constant $\rho_T$. In section \ref{otherorbifolds} we show that $\rho_T$ is a lower bound for the maximum injectivity radius of any orbifold of signature different from $(0,3)$ and the final section contains the proofs of theorems $1$ and $2$.\\
In the Appendix we show how to use the techniques developed in section \ref{triangularsurfaces}, to deduce a simple proof of Yamada's theorem described above.

\subsection*{Acknowledgements}
The author would like to thank her advisor Hugo Parlier for introducing her to the subject, for many helpful discussions and advice and for carefully reading the drafts of the paper.

\section{Notation and preliminary results}\label{preliminaries}
A {\it (hyperbolic)  surface} is a smooth connected orientable surface with a complete and finite area hyperbolic metric. A {\it cone-surface} is a two-dimensional connected manifold that can be triangulated by finitely many hyperbolic triangles; if it has boundary components, we ask that they be geodesic. A point where the cone-surface isn't smooth is called a {\it cone point}. Note that we will distinguish between cone points and cusps.

For every cone point $p$, there is a collection of triangles having $p$ as a vertex; we call {\it total angle} at $p$ the sum of angles at $p$ of those triangles. We will consider cone-surfaces such that every cone point of a cone-surface has total angle at most $\pi$, and we will call them {\it admissible cone-surfaces}. A cone-surface is of {\it signature} $(g,n,b)$ if it has genus $g$, $n$ singular points (cone points or cusps) and $b$ boundary components. If $b=0$, we will simply write $(g,n)$.

If a cone point has total angle $\frac{2\pi}{k}$, we say that it is {\it of order $k$}. We define an {\it orbifold} to be an admissible cone-surface without boundary components such that every cone point is of order $k$, for a positive integer $k$ (depending on the point). We denote by $\mathcal{O}$ the set of all orbifolds and by $\mathcal{S}$ the set of all surfaces.

We will use $d_M(\ , \ )$ for the distance on a metric space $M$ and $B_r(p)$  for the set of points at distance at most $r$ from a point $p\in M$: $$B_r(p)=\{q\in M : d_M(p,q)<r\}.$$ Sometimes, if the metric space we are considering is clear from the context, we will simply use $d(\ , \ )$ instead of $d_M(\ , \ )$. Given a curve $\gamma$ on $M$, we denote its length by $l(\gamma)$.

Let $S$ be an admissible cone-surface; we can consider the map
\begin{align*}
r:S &\rightarrow \mathbb{R}\\
p &\mapsto r_p,
\end{align*}
where $$r_p=\max\{r\geq 0\mbox{ }| B_r(p)\mbox{ is isometric to an open disk of radius $r$ in }\mathbb{H}^2\}$$
is the injectivity radius at the point $p$. We will need the following result.

\begin{prop}
The map $r$ admits a maximum.
\end{prop}

\begin{proof}
Let $\varepsilon>0$ be small. If $d(p,q)<\varepsilon$, then $r_q\geq r_p-\varepsilon$, because $B_{r_p-\varepsilon}(q)$ is embedded in $B_{r_p}(p)$; conversely $r_p\geq r_q-\varepsilon$. So $\left| r_p-r_q\right|\leq \varepsilon$ and the map $r$ is continuous.

If there is no cusp, $S$ is compact and $r$ has a maximum. If there are any cusps, note that the injectivity radius becomes smaller while getting closer to the cusp. So $$\sup_{p\in S}{r_p}=\sup_{p\in C}{r_p}$$ for some compact set $C$ in $S$ obtained by taking away suitable open horoballs. Again, we get that $r$ has a maximum.
\end{proof}

We define the {\it maximum injectivity radius} of an admissible cone-surface $S$ as
$$r(S):=\max_{p\in S}r_p.$$
The constant $\rho_T$ we want to find is the infimum of $r(O)$ for $O\in\mathcal{O}$. As every quotient surface is an orbifold and every orbifold is the quotient of a surface by a group of automorphisms, we have
$$\inf_{O\in\mathcal{O}}r(O)= \inf_{S\in\mathcal{S}}r(S/\mathrm{Aut}^+(S)).$$
We will prove that the two infima are realized by the orbifold of genus zero with three cone points of order $2$, $3$ and $7$, which can be obtained as quotient $S/\mathrm{Aut}^+(S)$ for any Hurwitz surface $S$.\\

We now want to obtain some results about geodesic representatives of simple closed curves and pants decompositions, in the context of admissible cone-surfaces. We will use a result by Tan, Wong and Zhang \cite{tanwongzhang}, while other results will be similar to some obtained by Dryden and Parlier in \cite{drydenparlier}, but the presence of cone points of order two will require some extra work.

Following \cite{drydenparlier}, we define a {\it (generalized) pair of pants} as for compact hyperbolic surfaces, with the difference that we allow the boundary geodesics to be replaced by cusps or cone points with total angle at most $\pi$. A pair of pants will be called:
\begin{itemize}
\item a {\it Y-piece} if it has three geodesics as boundary,
\item a {\it V-piece} if it has two geodesics and a singular point as boundary,
\item a {\it joker's hat} if it has two singular points and a geodesic as boundary,
\item a {\it triangular surface} if it has three singular points as boundary.
\end{itemize}

An {\it admissible geodesic of the first type} is a simple closed geodesic. An {\it admissible geodesic of the second type } is a curve obtained by following back and forth a simple geodesic path between two cone points of order two. These admissible geodesics will have here the same role played by simple closed geodesics in the context of hyperbolic surfaces: they will be the geodesic representatives of simple closed curves and they will form pants decompositions.

Let $\Sigma$ be the set of all cone points on an admissible cone-surface $S$. We say that a closed curve $\gamma$ on $S\setminus \Sigma$ is homotopic to a point $p\in S$ if $\gamma$ and $p$ are are freely homotopic on $S\setminus \Sigma\cup \{p\}$; $\gamma$ is {\it non-trivial} if it is not homotopic to any point (including the singular ones) or boundary component. Given an admissible geodesic of the second type $\delta$, let $\delta_\varepsilon$ be the boundary of an $\varepsilon$-neighborhood of $\delta$ (with $\varepsilon$ small enough so that $\delta_\varepsilon$ is a simple closed curve, contractible in $S\setminus \Sigma\cup\delta$). A curve $\gamma$ on $S\setminus \Sigma$ is homotopic to $\delta$ if it is freely homotopic on $S\setminus \Sigma$ to $\delta_\varepsilon$.

\begin{figure}[h]
\begin{center}
\includegraphics{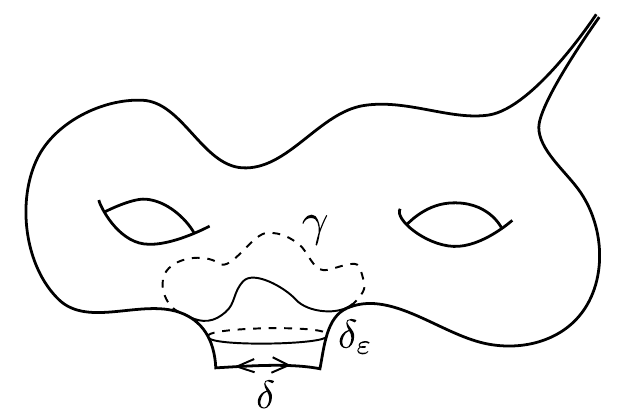}
\caption{A curve $\gamma$ homotopic to an admissible geodesic of the second type $\delta$}
\end{center}
\end{figure}

Let $\alpha$ and $\beta$ be  two transverse simple closed curve on $S\setminus\Sigma$ The {\it (geometric) intersection number} of the corresponding free homotopy classes is
$$i([\alpha],[\beta])=\min\{|\alpha'\cap\beta'| : \alpha'\in[\alpha], \beta'\in[\beta]\}.$$
We say that $\alpha$ and $\beta$  form a {\it bigon} if there is an embedded disk on $S\setminus \Sigma$ whose boundary is the union of an arc of $\alpha$ and an arc of $\beta$ intersecting in exactly two points.

In \cite{tanwongzhang}, the following is proven:
\begin{prop}\label{geodrepresentatives}
Let $S$ be an admissible cone-surface, $\Sigma$ the set of all cone points on $S$ and $\gamma$ a non-trivial simple closed curve on $S\setminus \Sigma$. Then it is homotopic to a unique admissible geodesic $\mathcal{G}(\gamma)$.
\end{prop}
We also have:
\begin{prop}\label{geodrepresentatives2}
Let $S$ be an admissible cone-surface and $\Sigma$ the set of all cone points on $S$. Given $\alpha$, $\beta$ two transverse non-trivial simple closed curves on $S\setminus \Sigma$, then either $\mathcal{G}(\alpha)=\mathcal{G}(\beta)$ or $|\mathcal{G}(\alpha)\cap\mathcal{G}(\beta)|\leq |\alpha\cap\beta |$.
\end{prop}

\begin{proof}
Suppose $\mathcal{G}(\alpha)$ and $\mathcal{G}(\beta)$ are distinct. Since they are geodesics, they do not form any bigons.

If they are both simple closed geodesics, by the bigon criterion (see for instance \cite[prop. 1.7]{farbmargalit}) they intersect minimally and $$|\mathcal{G}(\alpha)\cap\mathcal{G}(\beta)|=i([\alpha],[\beta])\leq |\alpha\cap\beta|.$$

If $\mathcal{G}(\alpha)$ is a closed geodesic of the second type and $\mathcal{G}(\beta)$ is a simple closed geodesics, we can choose a curve $\alpha'\in[\alpha]$ in a small neighborhood of $\mathcal{G}(\alpha)$ such that it doesn't form any bigons with $\mathcal{G}(\beta)$, so $|\alpha'\cap\mathcal{G}(\beta)|=i([\alpha],[\beta])\leq |\alpha\cap\beta|$. Since every intersection of $\mathcal{G}(\alpha)$ and $\mathcal{G}(\beta)$ corresponds to at least two intersections of $\alpha'$ and $\mathcal{G}(\beta)$, we have $|\mathcal{G}(\alpha)\cap\mathcal{G}(\beta)|\leq|\alpha'\cap\mathcal{G}(\beta)|\leq|\alpha\cap\beta|$.

\begin{figure}[h]
\begin{center}
\includegraphics{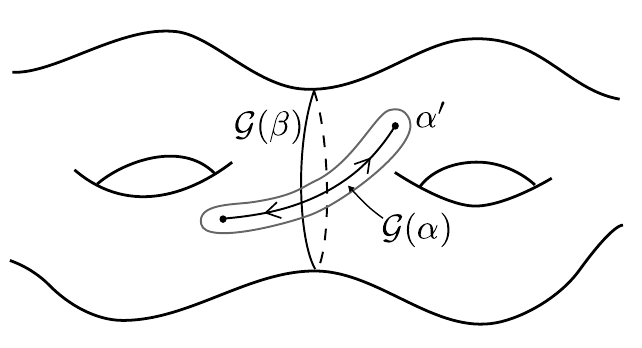}
\caption{$\mathcal{G}(\alpha)$, $\mathcal{G}(\beta)$ and the curve $\alpha'$}
\end{center}
\end{figure}

If both $\mathcal{G}(\alpha)$ and $\mathcal{G}(\beta)$ are closed geodesics of the second type, we consider two curves $\alpha'$ and $\beta'$ in small neighborhoods of the geodesics and we apply a similar argument to the one above.

\end{proof}

A maximal (in the sense of subset inclusion) set $\mathcal{P}$ of pairwise disjoint admissible geodesics is called a {\it pants decomposition}. 

\begin{prop}\label{pantsdecomposition}
Let $S$ be an admissible cone-surface of signature $(g,n,b)$ and $\mathcal{P}$ a pants decomposition of $S$. Then $|\mathcal{P}|=3g-3+n+b$ and $S\setminus \mathcal{P}$ is a set of (open) pairs of pants.
\end{prop}
The proof follows from Propositions \ref{geodrepresentatives} and \ref{geodrepresentatives2} and the fact that a topological surface of genus $g$ with $n+b$ punctures has $3g-3+n+b$ pairwise disjoint and non-homotopic simple closed curves.

Note that the number of curves for a pants decomposition is determined by the signature, but the number of pairs of pants is not. For example, we can consider a sphere with four cone points, two of order $2$ and two of order $3$. We can choose two different curves, each forming a pants decomposition; one decomposes the orbifold into two pairs of pants and the other into one, as in the following picture:
\begin{figure}[H]
\begin{center}
\includegraphics{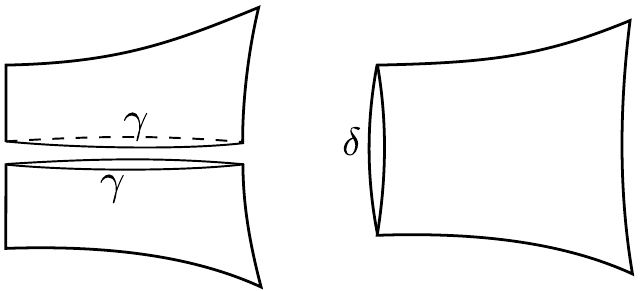}
\caption{Two pants decompositions with different number of pairs of pants}
\end{center}
\end{figure}

\section{Triangular surfaces}\label{triangularsurfaces}
A large portion of the work to obtain our result lies in the analysis of the triangular surfaces case. We will study the maximum injectivity radius of triangular surfaces and we will give an implicit formula for it. We will then prove that the triangular surface with cone points of order $2$, $3$ and $7$ realizes the minimum among all orbifolds of signature $(0,3)$.

Recall that a triangular surface is an admissible cone-surface of signature $(0,3)$. Every triangular surface can be obtained by gluing two hyperbolic triangles of angles $\alpha, \beta$ and $\gamma$, all less or equal to $\frac{\pi}{2}$; we denote the corresponding admissible cone-surface by $S_{\alpha, \beta, \gamma}$. We call $A$, $B$ and $C$ the vertices of the triangles corresponding to $\alpha$, $\beta$ and $\gamma$ respectively and $a$, $b$ and $c$ the opposite sides (or the lengths of the sides). Such a decomposition in triangles in unique, and since every triangle is uniquely determined (up to isometry) by its angles, the moduli space of triangular surfaces can be seen as $$\mathcal{M}=\left.\left\{(\alpha_1,\alpha_2,\alpha_3)\in\left[0,\frac{\pi}{2}\right]^3 :\sum_{i=1}^3\alpha_i<\pi\right\}\right/\mbox{Sym}(3),$$
 where the action of the symmetric group $\mbox{Sym}(3)$ on the set of triples is given by
$$(\sigma,(\alpha_1,\alpha_2,\alpha_3))\mapsto (\alpha_{\sigma(1)},\alpha_{\sigma(2)},\alpha_{\sigma(3)}).$$

\begin{lemma}
If a triangular surface $S$ contains a cusp and two singular points with total angles $2\vartheta_1$ and $2\vartheta_2$, the associated horocycle of length $h(S)$ is embedded in $S$, where 
$$h(S):=\frac{4}{\sqrt{1+\frac{1}{R(0,\vartheta_1,\vartheta_2)}}}$$
and
$$R(0,\vartheta_1,\vartheta_2):=\arctanh\frac{\cos\vartheta_1+\cos\vartheta_2}{2\sqrt{(1+\cos\vartheta_1)(1+\cos\vartheta_2)}}$$
is the radius of the inscribed disk in a triangle with angles $0$, $\vartheta_1$ and $\vartheta_2$.
\end{lemma}

\begin{proof}
Consider $T$, one of the two triangles that form $S$ and let $p$ be the center of the inscribed disk in $T$. Consider the horocycle passing through the points on $b$ and $c$ at distance $R(0,\vartheta_1,\vartheta_2)$ from $p$. This is an embedded horocycle and direct computation shows that it is longer than $h(S)$.
\end{proof}
\noindent {\sc Notation:} if $A$ (resp.\ $B$, $C$) is a cusp, we denote $h_A$ (resp.\ $h_B$, $h_C$) the associated horocycle of length $h(S)$.

\subsection{The maximum injectivity radius for triangular surfaces}\label{computations}\ \\
The aim of this section is to give an implicit formula for the maximum injectivity radius of any triangular surface. To obtain this result, we characterize points whose injectivity radius is maximum on a given triangular surface and we get a system of equations, including the desired implicit formula.

\subsubsection{Some special loops}\ \\
In this section we denote $S_{\alpha, \beta,\gamma}$ simply by $S$, since the angles will be fixed. For $p\in S$, we consider the maximal embedded disk $B_{r_p}(p)$. As $r_p$ is the injectivity radius in $p$, either the disk is tangent to itself in at least a point or a cone point of order two belongs to the boundary of the disk.

Every point where $B{r_p}(p)$ is tangent to itself corresponds to a simple loop on the surface, made by the two radii from the center $p$ to the point of tangency. The loop is geodesic except in $p$, its length is $2r_p$ and it is length-minimizing in its class in $\pi_1(S\setminus\{A,B,C\},p)$.

If there is a cone point of order two, say $A$, and it belongs to the boundary of $B_{r_p}(p)$, we associate a loop $\gamma_A$ obtained by traveling from $p$ to $A$ and back on the length realizing geodesic between these two points. Clearly, this loop has length $2r_p$.
\begin{figure}[h]
\begin{center}
\includegraphics{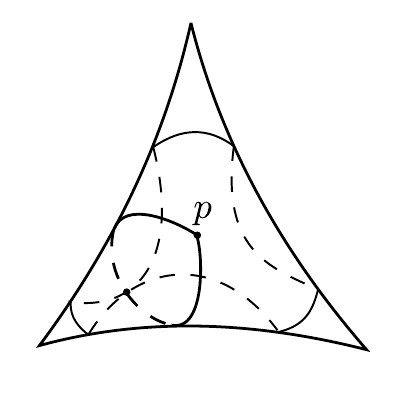}
\caption{A tangency point and its associated loop}
\end{center}
\end{figure}
We want to study the loops described above and use them to characterize points on $S$ with maximum injectivity radius.

Let $p$ be a non-singular point on $S$ and consider one corner, say $A$. Suppose $\alpha=\frac{\pi}{2}$; we define $\gamma_A$ to be, as before, the curve that traces the length realizing geodesic between $p$ and $A$ from $p$ to $A$ and back.
\begin{lemma}
Let $\mathcal{C}$ be the class in $\pi_1(S\setminus\{A,B,C\},p)$ of a simple loop around $A$ and based at $p$. Then
$$l(\gamma_A)=\inf\{l(\gamma)|\gamma\in \mathcal{C}\}.$$
\end{lemma}

\begin{proof}
Consider $\gamma\in \mathcal{C}$. If it crosses $\gamma_A$,  the two curves form a bigon and $\gamma$ can be shortened while staying in the same homotopy class. So, to compute $\inf\{l(\gamma)|\gamma\in \mathcal{C}\}$ we can assume $\gamma$ doesn't cross $\gamma_A$. We cut along $\gamma$ and $\gamma_A$ and we get a subset of $\mathbb{H}^2$ bounded by two curves between two copies of $P$: the first one is a geodesic given by $\gamma_A$ and the second one is given by $\gamma$, hence $l(\gamma)\geq l(\gamma_A)$. One can construct curves in the class with length arbitrarily close to $l(\gamma_A)$, so the infimum is $l(\gamma_A)$.
\end{proof}

\begin{figure}[h]
\begin{center}
\includegraphics{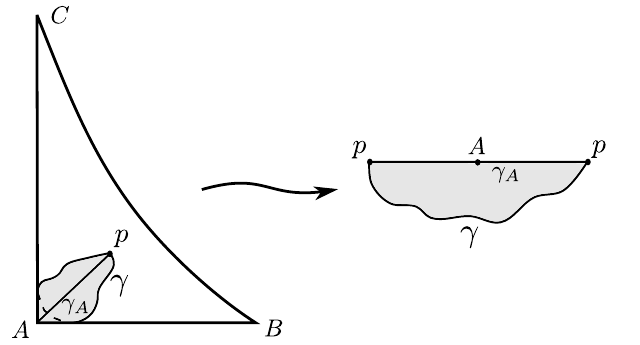}
\caption{Cutting along $\gamma_A$ for $A$ cone point of order two}
\end{center}
\end{figure}
Suppose now $\alpha<\frac{\pi}{2}$.

\begin{lemma}
There exists a unique simple loop $\gamma_A$ based at $p$, going around $A$, geodesic except at $p$. Moreover, it is length minimizing in its class in $\pi_1(S\setminus\{A,B,C\},p)$.
\end{lemma}

\begin{proof}
Consider the class $\mathcal{C}\in\pi_1(S\setminus\{A,B,C\},p)$ of a simple closed curve with base point $p$ that goes around $A$. We first show the uniqueness of $\gamma_A$. Suppose then $\gamma_A$ is a simple loop in $\mathcal{C}$, geodesic except at $p$. Then it doesn't cross the length-realizing geodesic between $p$ and $A$, otherwise they would form a geodesic bigon. We can cut along the geodesic; $\gamma_A$ is given by the unique geodesic between the two copies of $p$ on the cut surface.

We show now that $\gamma_A$ exists and is length-minimizing. Let $\gamma_n$ be a sequence of smooth curves in $\mathcal{C}$ with lengths converging to $\inf\{l(\gamma)|\gamma\in \mathcal{C}\}$. If there is any cusp on $S$, since the length of the curves $\gamma_n$ is bounded above, we can assume that they are contained in a compact subset of $S$. By the Arzel\`a-Ascoli Theorem (see \cite[Theorem A.19]{buser}) we get a limit curve $\gamma_A$. Note that $\gamma_A$ doesn't contain $A$: it is clear if $A$ is a cusp, while if $\alpha>0$ and $\gamma_A$ passes through $A$, it forms an angle at $A$ smaller or equal to $2\alpha<\pi$, so it can be shortened to a curve in $\mathcal{C}$, contradiction. In particular, $\gamma_A\in \mathcal{C}$. By the minimality, $\gamma_A$ is simple and geodesic except at $p$.
\end{proof}

\noindent{\sc Notation:} $l_A:=l(\gamma_A)$, $l_B:=l(\gamma_B)$ and $l_C:=l(\gamma_C)$. By $\tilde{\alpha}$ (resp. $\tilde{\beta}$, $\tilde{\gamma}$) we denote the acute angle of $l_A$ (resp. $l_B$, $l_C$) at $p$.

Clearly, for $\alpha=\frac{\pi}{2}$, the length is twice the distance $d(p,A)$. So $l_A$ increases (continuously) when $d(p,A)$ increases.

If $\alpha=0$, i.e.\ if $A$ is a cusp, consider the horocycle $h_A$. Cut along the loop itself and the geodesic from $p$ to the cusp. We represent the triangle we obtain in the upper half plane model of $\mathbb{H}^2$, choosing $A$ to be the point at infinity. Note that the geodesic between $p$ and the cusp is the bisector of $\tilde{\alpha}$. By hyperbolic trigonometry, we get
$$1=\cos(0)=\cosh\left(\frac{l_A}2\right)\sin\left(\frac{\tilde{\alpha}}{2}\right).$$
Let $d_{h_A}(p)$ be $d(p,h_A)$, if $p$ doesn't belong  to the horoball bounded by $h_A$, and $-d(p,h_A)$ otherwise. By direct computation we get that $l_A$ is a continuous monotone decreasing function of $\tilde{\alpha}$, so a continuous monotone increasing function of $d_{h_A}(p)$.

\begin{figure}[h]
\begin{center}
\includegraphics{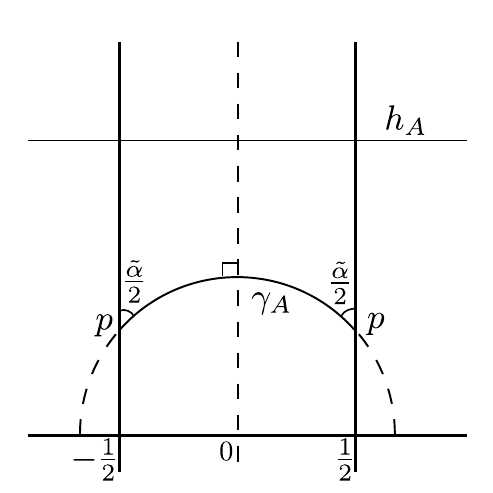}
\caption{The loop around a cusp in the upper half plane}
\end{center}
\end{figure}

If $\alpha\in\left(0,\frac{\pi}{2}\right)$, we cut along the geodesic from $p$ to $A$ and we get an isosceles triangle. Again, the geodesic from $p$ to $A$ is a bisector of $\tilde{\alpha}$. Using hyperbolic trigonometry, we obtain the equations
$$\cos(\alpha)=\cosh\left(\frac{l_A}2\right)\sin\left(\frac{\tilde{\alpha}}{2}\right)$$
and
$$\sinh\left(\frac{l_A}2\right)=\sin(\alpha)\sinh(d(p,A)).$$
In particular, $l_A$ is a continuous monotone increasing function of $d(p,A)$.
\begin{figure}[H]
\begin{center}
\includegraphics{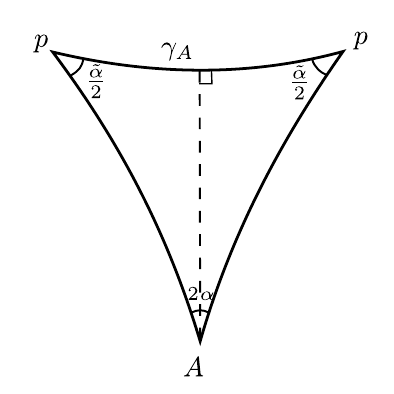}
\caption{The triangle associated to the loop for $\alpha\in\left(0,\frac{\pi}{2}\right)$}
\end{center}
\end{figure}
So, for $p\in S$, we can define
$$\dist_A(p)=\left\{\begin{array}{l}
d(p,A) \mbox{ if } \alpha\neq 0\\
d_{h_A}(p) \mbox{ if } \alpha= 0
\end{array}\right.$$
Similarly for $B$ and $C$. We have just proven the following:

\begin{lemma}
The length $l_A$ (resp.\ $l_B$, $l_C$) is a continuous monotone increasing function of $\dist_A(p)$ (resp.\ $\dist_B(p)$, $\dist_C(p)$).
\end{lemma}

Now, given a point $p$ on the surface, we can cut along the loops $\gamma_A$, $\gamma_B$ and $\gamma_C$. If there is no cone point of order two, we obtain four pieces: three of them are associated each to a singular point, and the fourth one is a triangle of side lengths $l_A$, $l_B$ and $l_C$.
\begin{figure}[H]
\begin{center}
\includegraphics{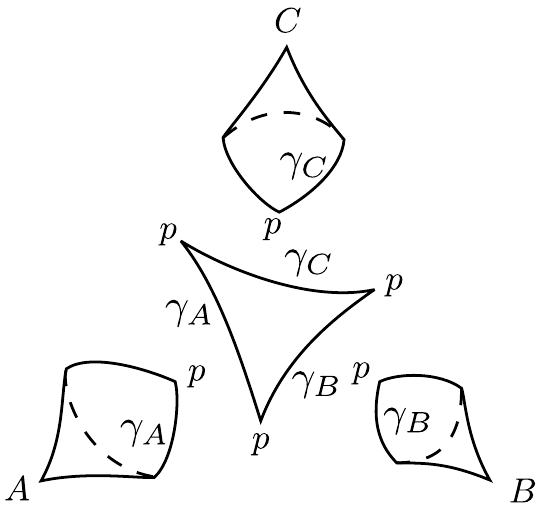}
\caption{$S$ cut along the three loops (case without cone points of order two)}
\end{center}
\end{figure}
If there is a cone point of order $2$, we get only three pieces: two associated to the other singular points and again a triangle of side lengths $l_A$, $l_B$ and $l_C$.
\begin{figure}[H]
\begin{center}
\includegraphics{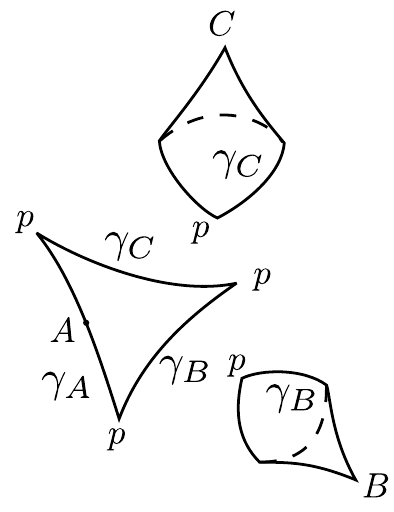}
\caption{$S$ cut along the three loops for $A$ of order $2$}
\end{center}
\end{figure}
\subsubsection{Lengths of the loops and injectivity radius}\ \\
The aim of this part is to relate the lengths $l_A$, $l_B$ and $l_C$ with the injectivity radius in a point. We first need two lemmas.
\begin{lemma}\label{minloops}
Given a triangular surface $S$, a non-singular point $p\in S$ and its three associated loops, we have
$$r_p=\min\left\{\frac{l_A}{2},\frac{l_B}{2},\frac{l_C}{2}\right\}.$$
\end{lemma}

\begin{proof}
Denote $m:=\min\left\{\frac{l_A}{2},\frac{l_B}{2},\frac{l_C}{2}\right\}$.

We already remarked that either $B_{r_p}(p)$ is tangent to itself in at least a point or its boundary contains a cone point of order $2$. In both situations, we get a loop at $p$ of length $2r_p$. So $r_p\geq m$.

To prove the other inequality, suppose $\gamma_A$ is the shortest of the three loops, i.e.\ $l_A=2m$. If $\alpha=\frac{\pi}{2}$, then $r_p\leq d(p,A)=\frac{l_A}{2}=m$, since an embedded disk cannot contain a singular point. Otherwise, if $\alpha\neq \frac{\pi}{2}$, consider the point $q$ on $\gamma_A$ at distance $\frac{l_A}{2}=m$ from $p$. Every disk of radius bigger than $\frac{l_A}{2}$ overlaps in $q$, hence again $r_p\leq m$.
\end{proof}
Let $T\subseteq \mathbb{H}^2$ be a hyperbolic triangle with angles $\alpha,\beta,\gamma$ at most $\frac{\pi}{2}$. If $\alpha=0$, consider a horocycle $H_A$ based at $A$ such that $$l(H_A\cap T)=\frac{1}{2}h(\alpha,\beta,\gamma).$$
For $p\in \mathbb{H}^2$, we define $\dist_A(p)$, $\dist_B(p)$ and $\dist_C(p)$ as before in the case of a triangular surface, with $H_A$ instead of $h_A$.

\begin{lemma}\label{3dist}
Let $T\subseteq \mathbb{H}^2$ be a hyperbolic triangle and $p,p'\in T$. If 
$$\left\{\begin{array}{c}
\dist_A(p')\geq \dist_A(p)\\
\dist_B(p')\geq \dist_B(p)\\
\dist_C(p')\geq \dist_C(p)
\end{array}\right.$$
then $p=p'$.
\end{lemma}

\begin{proof}
Consider a corner, say $A$ (similarly for $B$ and $C$): we define 
\begin{gather*}\mathcal{C}_A:=\{q\in\mathbb{H}^2 \mbox{ }| \dist_A(q)=\dist_A(p)\}\\
\mathcal{D}_A:=\{q\in\mathbb{H}^2 \mbox{ }| \dist_A(q)\leq\dist_A(p)\}.
\end{gather*}
So $p\in\mathcal{C}_A\cap\mathcal{C}_B\cap\mathcal{C}_C\subseteq\mathcal{D}_A\cap\mathcal{D}_B\cap\mathcal{D}_C$.

Note that if $A$ is not an ideal point, $\mathcal{C}_A$ is a circle and $\mathcal{D}_A$ is the disk bounded by $\mathcal{C}_A$. If $A$ is an ideal point, $\mathcal{C}_A$ the horocycle based at $A$ passing through $p$ and $\mathcal{D}_A$ is the associated horoball. In any case, $\mathcal{D}_A$ is a convex set.

Let $\mathcal{D}$ be the union $\mathcal{D}_A\cup \mathcal{D}_B\cup\mathcal{D}_C$. Since $\mathcal{D}$ is star-shaped with respect to $p$ and it contains the three sides, $T\subseteq \mathcal{D}$.

Consider now $p'\in T$ with
$$\left\{\begin{array}{c}
\dist_A(p')\geq \dist_A(p)\\
\dist_B(p')\geq \dist_B(p)\\
\dist_C(p')\geq \dist_C(p)
\end{array}\right.$$
Then $$p'\in T\setminus(\mathring{\mathcal{D}_A}\cup\mathring{\mathcal{D}_B}\cup\mathring{\mathcal{D}_C}).$$
Since $T\subseteq\mathcal{D}$ we have $p'\in\mathcal{C}_A\cap\mathcal{C}_B\cap\mathcal{C}_C$. We show that $\mathcal{C}_A\cap\mathcal{C}_B\cap\mathcal{C}_C$ is only one point, hence $p=p'$. Suppose by contradiction that the intersection contains two points. In the Poincar\'e disk model, $\mathcal{C}_A$, $\mathcal{C}_B$ and $\mathcal{C}_C$ are Euclidean circles, so if they intersect in two points, these should both belong to the boundary of $\mathcal{D}$. In particular $p\in T\cap\partial\mathcal{D}\subseteq\partial T$. Suppose without loss of generality $p\in c$; then the circles $\mathcal{C}_A$ and $\mathcal{C}_B$ are tangent in $p$, so they intersect only in one point. As a consequence, $\mathcal{C}_A\cap\mathcal{C}_B\cap\mathcal{C}_C$ can contain at most one point, a contradiction.
\end{proof}
We can now characterize points with maximum injectivity radius.

\begin{prop}\label{3tg}
Given a point $p\in S$, the following are equivalent:
\begin{enumerate}
\item $p$ is a global maximum point for the injectivity radius,
\item $p$ is a local maximum point for the injectivity radius,
\item the three loops at $p$ have the same length $2r_p$.
\end{enumerate}
Moreover, there are exactly two global maximum points.
\end{prop}
\begin{proof}
$(1)\Rightarrow (2)$ Clear.

$(2)\Rightarrow (3)$ By contradiction, suppose only one or two loops have length $2r_p$. Without loss of generality, assume $l_A\leq l_B \leq l_C$; by Lemma \ref{minloops}, $r_p=\frac{l_A}{2}$.

If only one loop has length $2r_p$, then $l_A<l_B\leq l_C$. Let us consider the geodesic segment between $p$ and $A$. Every point $p'$ on the prolongation of the geodesic segment after $p$ satisfies $\dist_A(p')>\dist_A(p)$. Thus every such $p'$ has a longer loop around $A$. By continuity of the lengths of the loops, for points that are sufficiently close to $p$ the loop around $A$ is still the shortest. These points have injectivity radius bigger than $r_p$, i.e.\ $p$ is not a local maximum point.

If two loops have length $2r_p$, we have $l_A=l_B<l_C$; we choose this time the orthogonal $l$ from $p$ to the side $c$. Every point $p'$ on $l$ that is further away from $c$ then $p$ satisfies $\dist_A(p')> \dist_A(p)$ and $\dist_B(p')> \dist_B(p)$.
\begin{figure}[h]
\begin{center}
\includegraphics{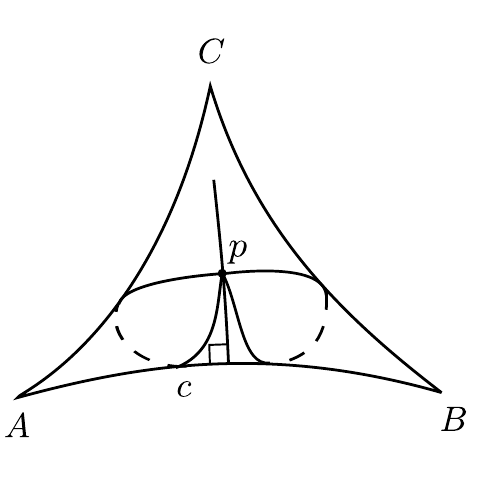}
\caption{$P$ with the orthogonal to $c$}
\end{center}
\end{figure}

So like before, every point on $l$ in a small enough neighborhood of $p$ has bigger injectivity radius than $r_p$ and again $p$ is not a local maximum point.

$(3)\Rightarrow (1)$ Suppose $l_A=l_B=l_C=2r_p$ and suppose by contradiction that $p$ isn't a global maximum point. Then there exists $q$ with $r_q>r_p$. Let $\psi$ be the orientation reversing isometry that exchanges the two triangles forming $S$ and fixes the sides pointwise. By possibly replacing $q$ with $\psi(q)$, we can assume $q$ belongs to the same triangle as $p$. Since $r_q>r_p$, the loops at $q$ are strictly longer than the ones at $p$. So $\dist_A(q)>\dist_A(p)$, $\dist_B(q)>\dist_B(p)$ and $\dist_C(q)>\dist_C(p)$, in contradiction with Lemma \ref{3dist}.

We still have to show that there are exactly two global maximum points. With similar arguments as the ones above, if two points $p$ and $p'$ on one triangle are global maximum points for the injectivity radius, then $p=p'$. So we have at most one global maximum point per triangle. To conclude that there are exactly two global maximum points, it's enough to show that a global maximum point cannot belong to a side. Suppose it does and assume, without loss of generality, $p\in c$. Since the geodesics from $A$ to $p$ and from $B$ to $p$ are bisectors of $\tilde{\alpha}$ and $\tilde{\beta}$ respectively, the loops $\gamma_A$ and $\gamma_B$ meet the sides $b$ and $a$ orthogonally. So $d(p,a)=d(p,b)=r_P$. Let $q$ and $q'$ be the intersections of $\gamma_C$ with $a$ and $b$. Then
$$2r_p=l_C>d(p,p)+d(p,q')\geq d(p,a)+d(p,b)=2r_p,$$
a contradiction.
\end{proof}

\subsubsection{Constructing the system}\ \\
We can now write a system of equations for a point with maximum injectivity radius. This will give us an implicit formula for $r(S)$.

Suppose first $\alpha,\beta,\gamma<\frac{\pi}{2}$, and let $p$ be a maximum for the injectivity radius. By Proposition \ref{3tg}, the three loops $\gamma_A$, $\gamma_B$ and $\gamma_C$ have length $2r_p$. Cutting along the three loops we get four pieces; the three associated to the singular points give us the equations
\begin{gather*}
\cos(\alpha)=\cosh(r_p)\sin\left(\frac{\tilde{\alpha}}{2}\right)\\
\cos(\beta)=\cosh(r_p)\sin\left(\frac{\tilde{\beta}}{2}\right)\\
\cos(\gamma)=\cosh(r_p)\sin\left(\frac{\tilde{\gamma}}{2}\right).
\end{gather*}
Vice versa, any positive solution to these equations determines the three pieces.

The fourth piece is a triangle, equilateral since the three sides are the three loops. Denote $\vartheta$ the angle of the triangle. By hyperbolic trigonometry we have
$$\cosh(2r_p)=\frac{\cos(\vartheta)+\cos^2(\vartheta)}{\sin^2(\vartheta)},$$
or equivalently
$$\cosh(r_p)=\sqrt{\frac{1}{2(1-\cos(\vartheta))}}.$$
A positive solution of the above equation determines the triangle. Note that $\vartheta\neq 0$, because otherwise two loops would have the same direction in $p$, which is impossible.

Since $p$ is not a singular point, the angles at it should sum up to $2\pi$. Thus, we obtain the following system:
\begin{equation}\label{sys1}
\left\{\begin{array}{l}
\cos(\alpha)=\cosh(r_p)\sin\left(\frac{\tilde{\alpha}}{2}\right)\\
\cos(\beta)=\cosh(r_p)\sin\left(\frac{\tilde{\beta}}{2}\right)\\
\cos(\gamma)=\cosh(r_p)\sin\left(\frac{\tilde{\gamma}}{2}\right)\\
\cosh(r_p)=\sqrt{\frac{1}{2(1-\cos(\vartheta))}}\\
\tilde{\alpha}+\tilde{\beta}+\tilde{\gamma}+3\vartheta=2\pi
\end{array}\right.
\end{equation}
We want a solution that satisfies $r_p>0$, $\tilde{\alpha},\tilde{\beta},\tilde{\gamma}\in(0,\pi)$ and $\vartheta\in(0,\frac{\pi}{3})$. If we have such a solution to the system, we have four pieces that can be glued to form the surface $S$.

Under the conditions $r_p>0$, $\tilde{\alpha},\tilde{\beta}\in[0,\pi)$ and $\vartheta\in(0,\frac{\pi}{3})$, the system (\ref{sys1}) is equivalent to

\begin{equation}\label{sys2}
\left\{\begin{array}{l}
\sin\left(\frac{\tilde{\alpha}}{2}\right)=\frac{\cos(\alpha)}{\cosh(r_p)}\\
\sin\left(\frac{\tilde{\beta}}{2}\right)=\frac{\cos(\beta)}{\cosh(r_p)}\\
\cos(\vartheta)=1-\frac{1}{2\cosh^2(r_p)}\\
\tilde{\gamma}=2\pi-\tilde{\alpha}-\tilde{\beta}-3\vartheta\\
F_{\alpha,\beta,\gamma}(r_p)=0
\end{array}\right.\end{equation}
where
\begin{multline*}
f_{\alpha,\beta,\gamma}(x)=2\cos(\gamma)x^4-(3x^2-1)(\sqrt{x^2-\cos^2(\alpha)}\sqrt{x^2-\cos^2(\beta)}-\cos(\alpha)\cos(\beta))\\
-(x^2-1)\sqrt{4x^2-1}(\cos(\alpha)\sqrt{x^2-\cos^2(\beta)}+\cos(\beta)\sqrt{x^2-\cos^2(\alpha)})\end{multline*}
and $F_{\alpha,\beta,\gamma}=f_{\alpha,\beta,\gamma}\circ\cosh$.\\ \

When the triple $(\alpha,\beta,\gamma)$ is fixed and it is clear to which angles we are referring to, we will simply call $f$ (respectively $F$) the function $f_{\alpha,\beta,\gamma}$ (respectively $F_{\alpha,\beta,\gamma}$).

\begin{rmk}
The special role of $\gamma$ in the function $f_{\alpha,\beta,\gamma}$ follows from the choice to express $\tilde{\gamma}$ in terms of the other angles.
\end{rmk}
Suppose now that one angle, say $\alpha$, is a right angle. Then we have only three pieces obtained by cutting $S$ along the three loops. We have the following system:

\begin{equation}\label{sys1-right}
\left\{\begin{array}{l}
\cos(\beta)=\cosh(r_p)\sin\left(\frac{\tilde{\beta}}{2}\right)\\
\cos(\gamma)=\cosh(r_p)\sin\left(\frac{\tilde{\gamma}}{2}\right)\\
\cosh(r_p)=\sqrt{\frac{1}{2(1-\cos(\vartheta))}}\\
\tilde{\beta}+\tilde{\gamma}+3\vartheta=2\pi
\end{array}\right.
\end{equation}
Again, we want a solution that satisfies $r_p>0$, $,\tilde{\beta},\tilde{\gamma}\in(0,\pi)$ and $\vartheta\in(0,\frac{\pi}{3})$.
\begin{rmk} If we ask $r_p>0$, the equation
$$\cos(\alpha)=\cosh(r_p)\sin\left(\frac{\tilde{\alpha}}{2}\right)$$
has $\tilde{\alpha}=0$ as solution if and only if $\alpha=\frac{\pi}{2}$.
\end{rmk}
From the previous remark it follows that we can consider the same system (\ref{sys2}) for any surface. We look for a solution $r_p>0$, $\tilde{\alpha},\tilde{\beta},\tilde{\gamma}\in[0,\pi)$ and $\vartheta\in(0,\frac{\pi}{3})$.
\subsubsection{Existence and uniqueness of a solution}\ \\
We now prove the following:
\begin{prop}
There exists a unique solution to (\ref{sys2}) satisfying $r_p>0$, ${\tilde{\alpha},\tilde{\beta},\tilde{\gamma}\in[0,\pi)}$ and $\vartheta\in(0,\frac{\pi}{3})$.
\end{prop}
\begin{proof}
One can check that $f(1)>0$ and $\lim_{x\rightarrow +\infty}f(x)=-\infty.$ Since $f$ is continuous, it has a zero in $(1,+\infty)$, i.e.\ $F$ has a zero in $(0,+\infty)$.

To prove  that the whole system has a solution, we have to find non-negative angles satisfying the equations above. Given $r_p$ such that $F(r_p)=0$, there are unique $\frac{\tilde{\alpha}}{2},\frac{\tilde{\beta}}{2}\in\left[0,\pi/2\right)$ that satisfy
\begin{gather*}\sin\left(\frac{\tilde{\alpha}}{2}\right)=\frac{\cos(\alpha)}{\cosh(r_p)}\\
\sin\left(\frac{\tilde{\beta}}{2}\right)=\frac{\cos(\beta)}{\cosh(r_p)}.
\end{gather*}
Since
$$\cos(\vartheta)=1-\frac{1}{2\cosh^2(r_p)}\in\left(\frac{1}{2},1\right),$$
there is a unique solution $\vartheta\in\left(0,\frac{\pi}{3}\right)$. As
$$\tilde{\gamma}=2\pi-\tilde{\alpha}-\tilde{\beta}-3\vartheta\in (-\pi,2\pi),$$
we get $\frac{\tilde{\gamma}}{2}\in\left(-\frac{\pi}{2},\pi\right)$. Using the equivalence with system (\ref{sys1}), we know that
$$\sin\left(\frac{\tilde{\gamma}}{2}\right)=\frac{\cos(\gamma)}{\cosh(r_p)}\in[0,1),$$
thus $\tilde{\gamma}/2 \in [0,\pi)$. So, for every zero of $F$, we have a unique solution $(r_p,\tilde{\alpha},\tilde{\beta},\tilde{\gamma},\vartheta)$ to (\ref{sys2}) with the required conditions.

A solution gives us a radius and two associated centers, one for each triangle. We know, by Proposition \ref{3tg}, that there exist exactly two maximum points on $S$, which guarantees unicity of the solution of (\ref{sys2}).
\end{proof}
From the proof above we also see that $r(S)$ is the unique positive solution of $F=0$; for $r\geq 0$
$$F(r)>0 \Leftrightarrow r<r(S)$$ and $$F(r)<0 \Leftrightarrow r>r(S).$$
In terms of $f$, for $x\geq 1$ we have
$$f(x)>0 \Leftrightarrow x<\cosh(r(S))$$ and $$f(x)<0 \Leftrightarrow x>\cosh(r(S)).$$
\subsection{Continuity}\label{continuity}\ \\
In this section we want to show the continuity of $r(S_{\alpha,\beta,\gamma})$ with respect to the parameters $\alpha, \beta$ and $\gamma$.

Consider the function
\begin{align*}
\mathcal{F}:\mathcal{A}\times \mathbb{R}&\longrightarrow\mathbb{R}\\
(\alpha,\beta,\gamma,y)&\longmapsto \mathcal{F}(\alpha,\beta,\gamma,y):=F_{\alpha,\beta,\gamma}(y)
\end{align*}
where $\mathcal{A}$ is the set of possible triples of angles:
$$\mathcal{A}=\left\{(\alpha,\beta,\gamma)\in\left[0,\frac{\pi}{2}\right]^3 : \alpha+\beta+\gamma<\pi\right\}.$$
Note that $\mathcal{F}$ is continuous, as one can see from the explicit form of $F_{\alpha,\beta,\gamma}(y)$.
\begin{rmk}\label{maxradius}
Given a triangular surface $S$ and a point $p\in S$ with maximum injectivity radius, we have
$$\area\left( B_{r(S)}(p)\right) \leq \area(S)<2\pi.$$
So there is a constant $M>0$ such that $r(S)\leq M$ for all triangular surfaces.
\end{rmk}

\begin{prop}
The map
\begin{align*}
r(S_*):\mathcal{A}&\rightarrow \mathbb{R}\\
(\alpha,\beta,\gamma)&\mapsto r(S_{\alpha,\beta,\gamma})
\end{align*}
is continuous.
\end{prop}
\begin{proof}To prove the continuity, we show that for every sequence $$\{(\alpha_n,\beta_n,\gamma_n)\}_{n=1}^{+\infty}\subseteq \mathcal{A}$$ converging to some triple $(\alpha,\beta,\gamma)\in \mathcal{A}$, the limit $\displaystyle \lim_{n\rightarrow +\infty}r(S_{\alpha_n,\beta_n,\gamma_n})$ exists and it is $r(S_{\alpha,\beta,\gamma})$.

From Remark \ref{maxradius}, we know that $\{r(S_{\alpha_n,\beta_n,\gamma_n})\}_{n=1}^{+\infty}$ is contained in the compact set $[O,M]$, so it has an accumulation point $y$. Then there exists a subsequence $$\{r(S_{\alpha_{n_k},\beta_{n_k},\gamma_{n_k}})\}_{k=1}^{+\infty}$$ converging to $y$. Since $y\geq 0$ and $r(S_{\alpha,\beta,\gamma})$ is the only positive zero of $F_{\alpha,\beta,\gamma}$, we have
$$y=r(S_{\alpha,\beta,\gamma}) \Leftrightarrow \mathcal{F}(\alpha,\beta,\gamma,y)=F_{\alpha,\beta,\gamma}(y)=0.$$
We can compute $\mathcal{F}(\alpha,\beta,\gamma,y)$:
\begin{gather*}
\mathcal{F}(\alpha,\beta,\gamma,y)=\mathcal{F}\left(\alpha,\beta,\gamma,\lim_{k\rightarrow +\infty}r(S_{\alpha_{n_k},\beta_{n_k},\gamma_{n_k}})\right)=\\ \mathcal{F}\left(\lim_{k\rightarrow +\infty}\left(\alpha_{n_k},\beta_{n_k},\gamma_{n_k},r(S_{\alpha_{n_k},\beta_{n_k},\gamma_{n_k}})\right)\right)\stackrel{(\star)}{=}\\ \lim_{k\rightarrow +\infty}\mathcal{F}\left(\alpha_{n_k},\beta_{n_k},\gamma_{n_k},r(S_{\alpha_{n_k},\beta_{n_k},\gamma_{n_k}})\right)=\\ \lim_{k\rightarrow +\infty}F_{\alpha_{n_k},\beta_{n_k},\gamma_{n_k}}\left(r(S_{\alpha_{n_k},\beta_{n_k},\gamma_{n_k}})\right),
\end{gather*}
where $(\star)$ follows from the continuity of $\mathcal{F}$. Now for every $k$
$$F_{\alpha_{n_k},\beta_{n_k},\gamma_{n_k}}\left(r(S_{\alpha_{n_k},\beta_{n_k},\gamma_{n_k}})\right)=0,$$
so the limit is zero too and $y=r(S_{\alpha,\beta,\gamma})$.

Thus, every accumulation point of $\{r(S_{\alpha_n,\beta_n,\gamma_n})\}_{n=1}^{+\infty}$ is $r(S_{\alpha,\beta,\gamma})$, so the sequence converges and its limit is $r(S_{\alpha,\beta,\gamma})$.
\end{proof}

\subsection{The minimum among triangular orbifolds}\ \\
We will call {\it triangular orbifold} an orbifold of signature $(0,3)$. We want to find which one has smallest maximum injectivity radius.

\noindent {\sc Notation:} we will write $(\alpha,\beta,\gamma)\leq(\alpha',\beta',\gamma')$ for $\alpha\leq\alpha',\beta\leq \beta'$ and $\gamma\leq\gamma'$.

\begin{lemma}\label{angles}
If $(\alpha,\beta,\gamma)\leq(\alpha',\beta',\gamma')$, then $r(S_{\alpha,\beta,\gamma})\geq r(S_{\alpha',\beta',\gamma'})$ with equality if and only if $(\alpha,\beta,\gamma)=(\alpha',\beta',\gamma')$.
\end{lemma}

\begin{proof}
We pass from $(\alpha',\beta',\gamma')$ to $(\alpha,\beta,\gamma)$ in the following way
$$(\alpha',\beta',\gamma')\rightarrow (\alpha',\beta',\gamma)\rightarrow(\alpha',\beta,\gamma)\rightarrow(\alpha,\beta,\gamma).$$
and show that in any passage the radius decreases. Consider the first step. For all $x\geq 0$:
$$F_{\alpha',\beta',\gamma'}(x)-F_{\alpha',\beta',\gamma}(x)=2(\cos\gamma'-\cos\gamma)\cosh(x)^4\geq 0$$
because $\gamma'\leq \gamma$. So $F_{\alpha',\beta',\gamma'}(x)\geq F_{\alpha',\beta',\gamma}(x)$. Since $r(S_{\alpha',\beta',\gamma'})$ (resp.\ $r(S_{\alpha',\beta',\gamma})$) is the only solution in $(0,+\infty)$ of $F_{\alpha',\beta',\gamma'}$ (resp.\ $F_{\alpha',\beta',\gamma}$), we have $$r(S_{\alpha',\beta',\gamma'})\leq r(S_{\alpha',\beta',\gamma})$$
The same argument applied to all the steps shows that $$r(S_{(\alpha,\beta,\gamma)})\geq r(S_{(\alpha',\beta',\gamma')}).$$
\end{proof}

\begin{lemma}\label{tripleleq}
Given any triple of angles $\alpha,\beta,\gamma$ corresponding to a triangular surface $S_{\alpha,\beta,\gamma}$, $(\alpha,\beta,\gamma)$ is less or equal to (at least) one triple among $(\pi/2,\pi/3,\pi/7)$, $(\pi/2,\pi/4,\pi/5)$ and $(\pi/3,\pi/3,\pi/4)$.
\end{lemma}

\begin{proof}
Suppose $\alpha\geq\beta\geq\gamma$. By definition, all the angles are between $\pi/2$ and $0$. Since $\alpha+\beta+\gamma<\pi$, either $\alpha=\pi/2$ and $\beta,\gamma<\pi/2$, or $\alpha <\pi/2$.

If $\alpha=\pi/2$, the condition $\alpha+\beta+\gamma<\pi$ implies that $\beta\leq \pi/3$. If $\beta=\pi/3$, then $\gamma$ should be at most $\pi/7$, hence $(\alpha,\beta,\gamma)\leq(\pi/2,\pi/3,\pi/7)$. If $\beta<\pi/3$, i.e.\ $\beta\leq \pi/4$, then either $\beta=\pi/4$ and $\gamma\leq \pi/5$, or $\gamma\leq\beta\leq\pi/5$. In both situations, $(\alpha,\beta,\gamma)\leq(\pi/2,\pi/4,\pi/5)$.

If $\alpha<\pi/2$, then $\alpha,\beta,\gamma\leq\pi/3$. Moreover, since $\alpha+\beta+\gamma<\pi$ and $\gamma$ is the smallest angle, then $\gamma \leq\pi/4$. So $(\alpha,\beta,\gamma)\leq(\pi/3,\pi/3,\pi/4)$, and this ends the proof.
\end{proof}

From Lemmas \ref{angles} and \ref{tripleleq}, it follows that it suffices to compare $S_{\frac{\pi}{2},\frac{\pi}{3},\frac{\pi}{7}}$, $S_{\frac{\pi}{2},\frac{\pi}{4},\frac{\pi}{5}}$ and $S_{\frac{\pi}{3},\frac{\pi}{3},\frac{\pi}{4}}$ to find the triangular orbifold with smallest maximum injectivity radius.

To prove that $S_{\frac{\pi}{2},\frac{\pi}{3},\frac{\pi}{7}}$ is minimal, we first compute its maximum injectivity radius. We know that $\cosh\left(r(S_{\frac{\pi}{2},\frac{\pi}{3},\frac{\pi}{7}})\right)$ is the only solution (bigger than $1$) of $$F_{\frac{\pi}{2},\frac{\pi}{3},\frac{\pi}{7}}(x)=0,$$
from which we obtain
$$\cosh\left(r(S_{\frac{\pi}{2},\frac{\pi}{3},\frac{\pi}{7}})\right)=\sqrt{t_0},$$
where $t_0$ is the unique real solution of
$$\left(4-\cos^2\left(\frac{\pi}{7}\right)\right)t^3-5t^2+2t-\frac{1}{4}=0.$$
One can check that $$F_{\frac{\pi}{2},\frac{\pi}{4},\frac{\pi}{5}}(r(S_{\frac{\pi}{2},\frac{\pi}{3},\frac{\pi}{7}}))<0 \mbox{  and } F_{\frac{\pi}{3},\frac{\pi}{3},\frac{\pi}{4}}(r(S_{\frac{\pi}{2},\frac{\pi}{3},\frac{\pi}{7}}))<0,$$
so
$$r(S_{\frac{\pi}{2},\frac{\pi}{4},\frac{\pi}{5}})>r(S_{\frac{\pi}{2},\frac{\pi}{3},\frac{\pi}{7}}) \mbox{  and } r(S_{\frac{\pi}{3},\frac{\pi}{3},\frac{\pi}{4}})>r(S_{\frac{\pi}{2},\frac{\pi}{3},\frac{\pi}{7}}).$$

We have proven the following:
\begin{prop}\label{mintriang}
For every triangular orbifold $S$, $r(S)\geq r(S_{\frac{\pi}{2},\frac{\pi}{3},\frac{\pi}{7}})$, with equality if and only if $S=S_{\frac{\pi}{2},\frac{\pi}{3},\frac{\pi}{7}}$.
\end{prop}

We define $\rho_T$ to be $r\left(S_{\frac{\pi}{2},\frac{\pi}{3},\frac{\pi}{7}}\right)$. So
$$\rho_T=\arccosh\sqrt{t_0}$$
where $t_0$ is the unique real root of
$$\left(4-\cos^2\left(\frac{\pi}{7}\right)\right)t^3-5t^2+2t-\frac{1}{4}=0$$
and numerically
$$\rho_T\approx 0.187728.$$

\section{Orbifolds of signature $(g,n)\neq (0,3)$}\label{otherorbifolds}

Our goal now is to show that $\rho_T<r(S)$ for every non-triangular orbifold $S$. We will use the fact that Y-pieces and triangles of large enough area contain disks of radius bigger than $\rho_T$. We show that most of the non-triangular orbifolds contain a Y-piece or a big enough triangle and we analyze the remaining cases separately.

\subsection{Finding Y-pieces}\ \\
We use the following well known result:
\begin{prop}
Every open Y-piece contains two closed disks of radius $\rho_Y=\frac{\log 3}{2}$.
\end{prop}
One can check that $F_{\frac{\pi}{2},\frac{\pi}{3},\frac{\pi}{7}}(\rho_Y)<0$, hence $\rho_Y>\rho_T$. As a consequence, if we find an embedded open Y-piece in a surface $S$, we know that $r(S)>\rho_T$.
\begin{prop}\label{pants}
Let $S$ be an orbifold of signature $(g,n)$. Then $S$ contains an embedded open Y-piece if and only if $g>0$ and $3g+n\geq 5$ or $g=0$ and $n\geq 6$.
\end{prop}
\begin{proof}
$[\Rightarrow]$ To have an embedded Y-piece, we need:
\begin{itemize}
\item if $g>0$, at least two curves in a pants decomposition;
\item if $g=0$, at least three curves in a pants decomposition (since we cannot embed a one-holed torus).
\end{itemize}
Via Proposition \ref{pantsdecomposition}, the number of curves in a pants decomposition is $3g-3+n$, so we get the stated conditions.\\
$[\Leftarrow]$ Consider $S\setminus \Sigma$ as topological surface. We say that two pants decompositions $\mathcal{P}_1$ and $\mathcal{P}_2$ are joined by an elementary move (see \cite{hatcherthurston}) if $\mathcal{P}_2$ can be obtained from $\mathcal{P}_1$ by removing a single curve and replacing it by a curve that intersects it minimally.
\begin{figure}[h]
\begin{center}
\includegraphics{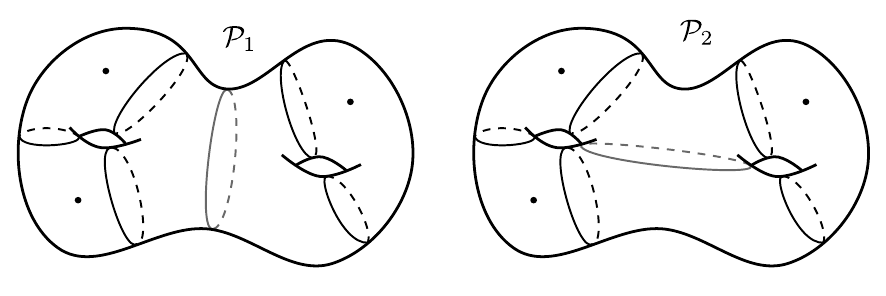}
\caption{Two pants decompositions joined by an elementary move}
\end{center}
\end{figure}
Given a pants decomposition, we can construct its dual graph $G$ (see \cite{buser}): vertices correspond to pairs of pants and there is an edge between two vertices if the corresponding pairs of pants share a boundary curve. In particular, there is a loop  at a vertex if two boundary curves of a pair of pants are glued to each other in the surface. The number of edges of $G$ is the number of curves in a pants decomposition, i.e.\ $3g-3+n$, and every vertex in $G$ has degree at most $3$. Suppose $v\in G$ is a vertex of degree $3$ and consider the associated pair of pants $P$. Passing to the geodesic representatives (via Propositions \ref{geodrepresentatives} and \ref{geodrepresentatives2}) of the boundary curves of $P$, we get a Y-piece, whose interior is embedded in $S$.

So, let us consider a surface $S$ satisfying $g>0$ and $3g+n\geq 5$ or $g=0$ and $n\geq 6$, a pants decomposition on it and its dual graph.

If $g\geq 2$, the graph has at least two circuits (that could be loops). The circuits are connected to each other, since $S$ is connected. So there is a vertex $v$ on one of the circuits connected to a vertex outside of the circuit. Thus the degree of $v$ is $3$ and, as seen before, we get an embedded open Y-piece.

If $g=1$, either the graph is a circuit with $3g-3+n\geq 2$ edges or it contains a circuit as a proper subgraph. In the first case, we can choose any edge and perform an elementary move on it to obtain a vertex of degree $3$. In the second case, there is a vertex $v$ on the circuit connected with a vertex outside the circuit, hence $\deg(v)=3$. Given a vertex of degree $3$ we have, as before, a Y-piece.
\begin{figure}[H]
\begin{center}
\includegraphics{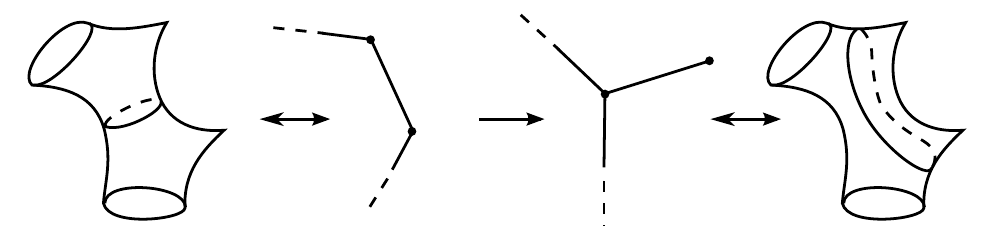}
\caption{An elementary move producing a vertex of degree $3$, for $g=1$}
\end{center}
\end{figure}
If $g=0$, the graph is a tree. If there is no vertex of degree three, the graph is a line with at least $3$ edges. Again we can perform an elementary move on an edge between two vertices of degree $2$ and get a vertex of degree $3$, so a Y-piece in $S$.
\end{proof}

\begin{cor}\label{withpants}
Every surface $S$ of genus $g$ with $n$ singular points such that $g>0$ and $3g+n\geq 5$ or $g=0$ and $n\geq 6$ satisfies $r(S)>\rho_T$.
\end{cor}

We are left with spheres with four or five singular points and tori with one singular point.

\subsection{Finding triangles in orbifolds of signature $(0,4)$ or $(0,5)$}\label{withtriangles}\ \\
Let $T$ be a hyperbolic triangle and $r(T)$ the radius of its inscribed disk, as in \cite[pp. 151--153]{beardon}. We know that
$$\tanh(r(T))\geq \frac{1}{2} \sin\left(\frac{1}{2} \area(T)\right).$$
It turns out that if $\area(T)\geq\frac{\pi}{4}$, we have $$F_{\frac{\pi}{2},\frac{\pi}{3},\frac{\pi}{7}}\left(\arctanh\left(\frac{1}{2} \sin\left(\frac{1}{2} \area(T)\right)\right)\right)<0,$$ so $r(T)>\rho_T$.

Given a sphere $S$ with $n$ singular points, $n=4$ or $5$, we will look for an embedded triangle of area at least $\frac{\pi}{4}$. Let $\frac{2\pi}{p_1},\dots\frac{2\pi}{p_n}$ be the angles at the singular points, where $p_i\geq 2$ is an integer or $p_i=\infty$ (i.e.\ the point is a cusp). Suppose $p_1\leq p_2\leq \dots \leq p_n$. The area of the surface is
$$\area(S)=(n-2)2\pi-\sum_{i=1}^{n}\frac{2\pi}{p_i}.$$
If we cut the surface into $2(n-2)$ triangles (by cutting it first into two $n$-gons and then cutting each polygon into $n-2$ triangles), the average area of the triangles is
$$\frac{\area(S)}{2(n-2)}=\pi\left(1-\frac{1}{n-2}\sum_{i=1}^n\frac{1}{p_i}\right).$$
This average is at least $\frac{\pi}{4}$ unless
\begin{enumerate}
\item $n=5$, $p_1=p_2=p_3=p_4=2$ and $p_5=2$ or $3$, or
\item $n=4$, $p_1=p_2=p_3=2$ and $p_4<\infty$ or $p_1=p_2=2$, $p_3=3$ and $p_4\leq 5$.
\end{enumerate}
So, if we are not in case (1) or (2), $S$ contains a triangle with area at least $\frac{\pi}{4}$, hence $r(S)>\rho_T$.

In case (1), cut the surface along geodesics as in the following picture:
\begin{figure}[H]
\begin{center}
\includegraphics{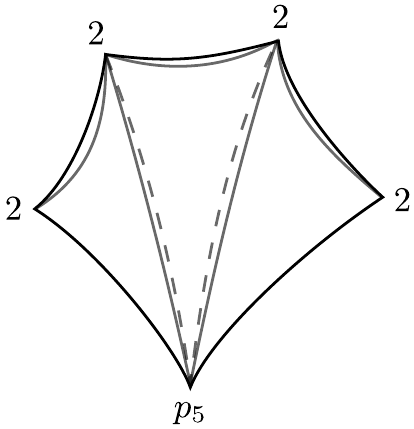}
\caption{A way to cut a $5$-punctured sphere into triangles}
\end{center}
\end{figure}
\noindent We obtain four triangles (if we cut open along a geodesic starting from a cone point of order two, we obtain an angle of $\pi$, so a side of a triangle). The average area of those triangles is
$$\frac{\area(S)}{4}=\frac{6\pi-\displaystyle\sum_{i=1}^{5}\frac{2\pi}{p_i}}{4}\geq \frac{\pi}{4},$$
so there is a triangle of area at least $\frac{\pi}{4}$, as desired.

In case (2), cut the surface along geodesics as in the following picture:
\begin{figure}[H]
\begin{center}
\includegraphics{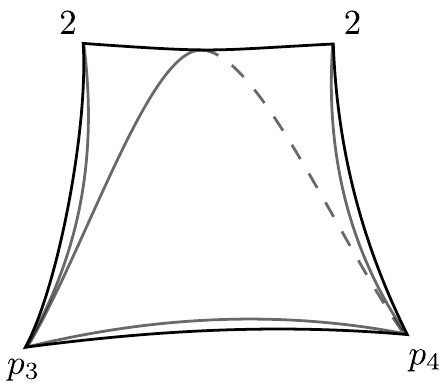}
\caption{A way to cut a $4$-punctured sphere into triangles}
\end{center}
\end{figure}
\noindent We have now two triangles and the average area is
$$\frac{\area(S)}{2}=\frac{4\pi-\displaystyle\sum_{i=1}^{4}\frac{2\pi}{p_i}}{2}\geq \pi-\frac{\pi}{p_3}-\frac{\pi}{p_4}.$$
If $p_3\geq 3$ or $p_4\geq 4$, this average is at least $\frac{\pi}{4}$ and we have a disk of large enough radius. The only case that is left is $p_1=p_2=p_3=2$ and $p_4=3$, which we will consider separately.

\subsection{Two special cases}\label{specialcases}\ \\
\noindent {\it First special case:} let $S$ be a sphere with three cone points of order $2$ and one of order $3$. Let $\tilde{S}_{\frac{\pi}{2},\frac{\pi}{3},\frac{\pi}{7}}$ be $S_{\frac{\pi}{2},\frac{\pi}{3},\frac{\pi}{7}}$ cut open along the side between the cone points of order $2$ and $7$. We want to embed $\tilde{S}_{\frac{\pi}{2},\frac{\pi}{3},\frac{\pi}{7}}$ into $S$.

We decompose $S$ into two isometric quadrilaterals by cutting along four geodesics between pairs of cone points. Each of the two quadrilaterals has three right angles and one angle $\frac{\pi}{3}$. We call $\lambda$ and $\mu$ the two sides which form the angle $\frac{\pi}{3}$ and $l$ (resp.\ $m$) the side opposite to $\lambda$ (resp.\ $\mu$). Using hyperbolic trigonometry, one can express $\lambda$ and $\mu$ as functions of $l$ and it is straightforward from the obtained expressions to deduce that $\lambda,\mu>\frac{\log 3}{2}$. Without loss of generality, assume $\lambda\geq\mu$. Since
$$\tanh(\lambda)\tanh(\mu)=\frac{1}{2},$$
we have 
$$\lambda\geq\arctanh\left(\frac{1}{\sqrt{2}}\right).$$
It is possible to check that 
\begin{gather*}
\frac{\log 3}{2}>c\\
\arctanh\left(\frac{1}{\sqrt{2}}\right)>a
\end{gather*}
where $a$, $b$ and $c$ are the sides opposite to $\frac{\pi}{2}$, $\frac{\pi}{3}$ and $\frac{\pi}{7}$ in one of the triangles forming $S_{\frac{\pi}{2},\frac{\pi}{3},\frac{\pi}{7}}$. So $\mu>c$ and $\lambda>a$ and we can embed the triangle in the quadrilateral by placing $B$ on the cone point of order three, $c$ on $\mu$ and $a$ on $\lambda$.

\begin{figure}[H]
\begin{center}
\includegraphics{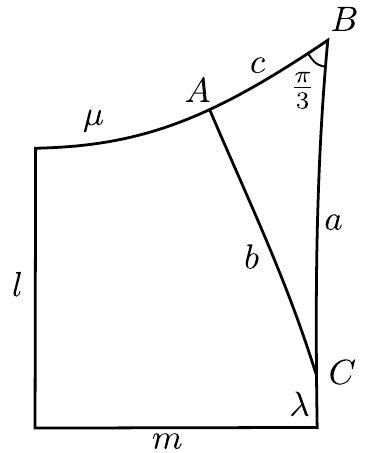}
\caption{The embedding of the triangle in the quadrilateral}
\end{center}
\end{figure}
By embedding in the same way another copy of the triangle in the other quadrilateral, we obtain the embedding of $\tilde{S}_{\frac{\pi}{2},\frac{\pi}{3},\frac{\pi}{7}}$ in $S$.

Given a point on $S_{\frac{\pi}{2},\frac{\pi}{3},\frac{\pi}{7}}$ that realizes the maximum injectivity radius, consider the corresponding point $p$ on $\tilde{S}_{\frac{\pi}{2},\frac{\pi}{3},\frac{\pi}{7}}\subseteq S$.
We know that $B_{\rho_T}(p)\cap \tilde{S}_{\frac{\pi}{2},\frac{\pi}{3},\frac{\pi}{7}}$ is embedded in $\tilde{S}_{\frac{\pi}{2},\frac{\pi}{3},\frac{\pi}{7}}$, hence in $S$.  Note that the distance of $p$ to the order two points is bigger than $d(p,m)$ or $d(p,l)$. So it's enough to prove that the distances $d(p,l)$ and $d(p,m)$ are strictly bigger than $\rho_T$ to deduce that $B_{\rho_T}(p)$ is embedded in $S$.\\
\begin{figure}[H]
\begin{center}
\includegraphics{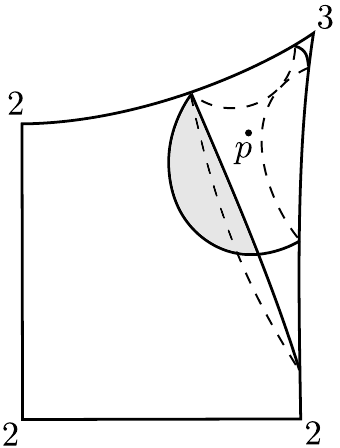}
\caption{The disk in the surface $S$}
\end{center}
\end{figure}
We have:
$$\mu-c>\frac{\log3}{2}-c$$
and $F_{\frac{\pi}{2},\frac{\pi}{3},\frac{\pi}{7}}\left(\frac{\log3}{2}-c\right)<0$, hence $\mu-c>\frac{\log3}{2}-c>\rho_T$. Similarly
$$\lambda-a\geq \arctanh\left(\frac{1}{\sqrt{2}}\right)-a$$
and 
$$F_{\frac{\pi}{2},\frac{\pi}{3},\frac{\pi}{7}}\left(\arctanh\left(\frac{1}{\sqrt{2)}}\right)-a\right)<0,$$
so
$$\lambda-a\geq\arctanh\left(\frac{1}{\sqrt{2}}\right)-a>\rho_T.$$
Since $d(p,l)>\mu-c>\rho_T$ and $d(p,m)>\lambda-a>\rho_T$, $B_{\rho_T}(p)$ is embedded in $S$.

This shows that $r(S)\geq\rho_T$. Actually, since $d(p,l)$ and $d(p,m)$ are strictly bigger than $\rho_T$, we can choose a point $p'$ in a small neighborhood of $p$ such that
$$\left\{\begin{array}{l}
d(p',l)>\rho_T\\
d(p',m)>\rho_T\\
d(p',B)>d(p,B),
\end{array}\right.$$
where $B$ is the cone point of order $3$. Since we are increasing the distance from $B$, with the same argument used for triangular surfaces we get that $r_{p'}>r_p=\rho_T$. So $r(S)>\rho_T$.

\noindent {\it Second special case:} tori with exactly one singular point. We will use the following result \cite{parlier}:

\begin{prop}\label{rhopentagon}
Let $P$ be a right-angled pentagon and $\mathring{P}$ its interior. Then $\mathring{P}$ contains a close disk of radius $\displaystyle{\rho_P=\frac{1}{2}\log\left(\frac{9+4\sqrt{2}}{7}\right)}$.
\end{prop}

Note that $\rho_P>\rho_T$, as $F_{\frac{\pi}{2},\frac{\pi}{3},\frac{\pi}{7}}(\rho_P)<0$. So again, the idea is to find a right-angled pentagon in any torus $\mathbb{T}$ with a singular point, to deduce that $r(\mathbb{T})>\rho_T$. We proceed as follows.

First of all, we choose a simple closed geodesic on $\mathbb{T}$ and we cut open along it; we obtain a V-piece. Following \cite{dianu}, this can be divided into two isometric pentagons with four right angles and an angle $\frac{\pi}{n}$ or $0$.
\begin{lemma}\label{rightangledpentagon}
Let $P$ be a pentagon with four right angles and an angle $\alpha\in[0,\frac{\pi}{2}]$. Then $P$ contains a right-angled pentagon.
\end{lemma}

\begin{proof}
If $\alpha=\frac{\pi}{2}$, the result is trivial. Let us suppose that $\alpha<\frac{\pi}{2}$.
\begin{figure}[H]
\begin{center}
\includegraphics{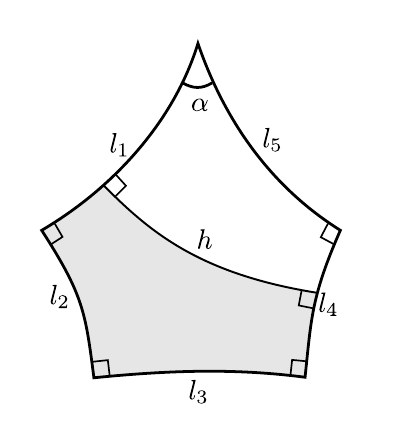}
\caption{The pentagon $\mathcal{P}$ with the embedded right-angled pentagon}
\label{pentagon}
\end{center}
\end{figure}
\noindent Label the sides of the pentagon as in figure \ref{pentagon} and consider $h$, the common orthogonal of $l_1$ and $l_4$. The pentagon bounded by $l_1$, $l_2$, $l_3$, $l_4$ and $h$ is a right-angled pentagon contained in $P$.
\end{proof}

With this result we can get the following.

\begin{lemma}
Let $\mathbb{T}$ be a torus with a singular point. Then $r(\mathbb{T})>\rho_T$.
\end{lemma}

\begin{proof}
As we described before, we can find in $\mathbb{T}$ a pentagon with four right angles and an angle $\alpha\in[0,\frac{\pi}{2}]$. Via Lemma \ref{rightangledpentagon} this pentagon contains a right-angled pentagon, which contains a disk of radius $$\rho_P>\rho_T$$
by Proposition \ref{rhopentagon}. So $$r(\mathbb{T})\geq \rho_P>\rho_T.$$
\end{proof}

\section{The minimum maximum injectivity radius for orbifolds}

Collecting all the results of the previous sections, we get our main result.

\begin{theor}\label{main}
For every orbifold $S$, the maximum injectivity radius satisfies
$r(S)\geq \rho_T$
and we have equality if and only if $S$ is the triangular surface $S_{\frac{\pi}{2},\frac{\pi}{3},\frac{\pi}{7}}$.
\end{theor}

\begin{proof}
Let $S$ be an orbifold of signature $(g,n)$; if it has positive genus and $3g-3\geq 5$ or if its genus is zero and it has at least six singular points, $r(S)>\rho_T$ by Corollary \ref{withpants}. By the results of sections \ref{withtriangles} and \ref{specialcases} the same holds for genus zero or one and $(g,n)\neq(0,3)$.  For $(g,n)=(0,3)$, by Proposition \ref{mintriang} we have $r(S)\geq \rho_T$ with equality if and only if $S=S_{\frac{\pi}{2},\frac{\pi}{3},\frac{\pi}{7}}$.
\end{proof}

Given a hyperbolic surface $S$, we define $\rho(S)$ to be the supremum of all $\rho$ such that there exists $p\in S$ with embedded and pairwise disjoint balls $$\{B_{\rho}(\varphi(p))\}_{\varphi\in \mathrm{Aut}^+(S)}.$$
As a corollary of the previous theorem, we give a sharp lower bound to $\rho(S)$, for $S\in\mathcal{S}$.

\begin{cor}
For every hyperbolic surface $S$, $\rho(S)\geq\rho_T$, with equality if and only if $S$ is a Hurwitz surface.
\end{cor}

\begin{proof}
For any surface $S$, we can consider the quotient $S/\mathrm{Aut}^+(S)$ and the canonical projection $\pi:S\rightarrow S/\mathrm{Aut}^+(S)$. 
Note that for every $p\in S$, the balls $$\{B_{\rho}(\varphi(p))\}_{\varphi\in \mathrm{Aut}^+(S)}$$ are pairwise disjoint if and only if $\rho\leq r_{\pi(p)}$. To maximize the radius $\rho$, we choose $q\in S/\mathrm{Aut}^+(S)$ such that $r_q=r(S/\mathrm{Aut}^+(S))$ and $p\in \pi^{-1}(q)$. $S$ is a Hurwitz surface if and only if $S/\mathrm{Aut}^+(S)=S_{\frac{\pi}{2},\frac{\pi}{3},\frac{\pi}{7}}$ and the result follows.
\end{proof}

\section*{Appendix}
In this Appendix we use the same techniques of the rest of the article to give a new proof of the following theorem by Yamada:
\begin{theor}{\bf (Yamada \cite{yamada})}
For every hyperbolic surface $S$
$$r(S)\geq \rho_S=\arcsinh(2/\sqrt{3})$$
with equality if and only if $S$ is the thrice-punctured sphere.
\end{theor}

Note that another proof of this theorem has been given by Gendulphe in \cite{gendulphe}.

\begin{rmk}
We will use the Collar Lemma (cf.\ \cite{buser}) and in particular the fact that if two simple closed geodesics $\alpha$ and $\beta$ intersect $n$ times, then
$$\alpha\geq 2 n w(\beta)$$
where $w(\beta)=\arcsinh\left(\frac{1}{\sinh(\beta/2)}\right)$ is half the width of the collar of $\beta$.
\end{rmk}

\begin{proof}
Fix $p\in S$ with $r_p=r(S)$; the disk $B_{r_p}(p)$ determines at least two loops based at $p$ of length $2r_p$. If there are exactly two loops of length $2r_p$ such that either at least one is homotopic to a cusp, or their geodesic representatives do not intersect, we are in the situation of Figure \ref{2loops}, where $\alpha$, $\beta$ or $\gamma$ can be cusps:
\begin{figure}[H]
\begin{center}
\includegraphics{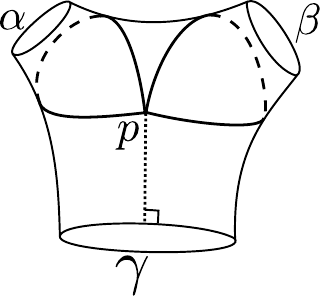}
\caption{The point $p$ with the two loops}\label{2loops}
\end{center}
\end{figure}
By choosing a point $q$ in a small enough neighborhood of $p$ on the orthogonal to the third curve, if $\gamma$ is not a cusps, or on the geodesic from $p$ to $\gamma$, if $\gamma$ is a cusp, we increase the lengths of the two loops around $\alpha$ and $\beta$ and by continuity all other loops are still longer. So $r_q>r_p$, a contradiction.

We have then two possibilities:
\begin{enumerate}[(a)]
\item there are two loops of length $2r_p$ whose geodesic representatives intersect each other once, or
\item there are at least three loops such that, if two are not homotopic to cusps, their geodesic representatives do not intersect.
\end{enumerate}

{\bf Case (a):} consider three loops of length $2r_p$; they determine a $3$-- or a $4$--holed sphere.

If there exists three loops determining a $3$--holed sphere, we can write equations for $p$. Denote by $\alpha$, $\beta$ and $\gamma$ the three boundary curves or cusps, by $\tilde{\alpha}$, $\tilde{\beta}$ and $\tilde{\gamma}$ the angles of the three loops at $p$ and by $\theta$ the angle of the (equilateral) triangle whose sides are the three loops (see Figure \ref{3loops}).
\begin{figure}[h]
\begin{center}
\includegraphics{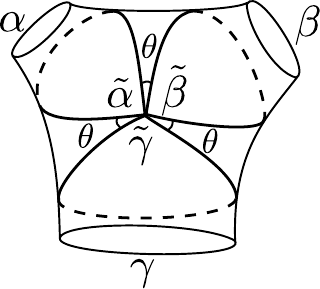}
\caption{Three loops determining a Y-piece}\label{3loops}
\end{center}
\end{figure}
We have:
\begin{equation}\label{system}
\left\{\begin{array}{l}
\cosh\left(\frac{\alpha}{2}\right)=\cosh(r_p)\sin\left(\frac{\tilde{\alpha}}{2}\right)\\
\cosh\left(\frac{\beta}{2}\right)=\cosh(r_p)\sin\left(\frac{\tilde{\beta}}{2}\right)\\
\cosh\left(\frac{\gamma}{2}\right)=\cosh(r_p)\sin\left(\frac{\tilde{\gamma}}{2}\right)\\
\cos(\theta)=1-\frac{1}{2\cosh^2(r_p)}\\
\tilde{\alpha}+\tilde{\beta}+\tilde{\gamma}+3\theta=2\pi
\end{array}\right.
\end{equation}
Using $r_p<\rho_S$, we get
$$\tilde{\alpha},\tilde{b},\tilde{\gamma}>2\arcsin\left(\sqrt{\frac{3}{7}}\right)$$
and
$$\theta>\arccos\left(\frac{11}{14}\right),$$
so 
$$\tilde{\alpha}+\tilde{\beta}+\tilde{\gamma}+3\theta>2\pi$$
which is a contradiction. Moreover, $r_p=\rho_S$ is a solution if and only if $\alpha=\beta=\gamma=0$, i.e. if we are on a thrice-punctured sphere. So in this case, $r(S)\geq\rho_S$, with equality if and only if $S$ is a thrice-punctured sphere.

If no three loops determine a $3$--holed sphere, fix three loops (with corresponding geodesic representatives or cusps $\alpha$, $\beta$ and $\gamma$) and the associated four-holed sphere (denote by $\delta$ the fourth boundary curve or cusp). The loop based at $p$ and homotopic to $\delta$ has length at least $2r_p$. We can again write down equations satisfied by the pieces we obtain by cutting the four-holed sphere along the loops. 
\begin{figure}[H]
\begin{center}
\includegraphics{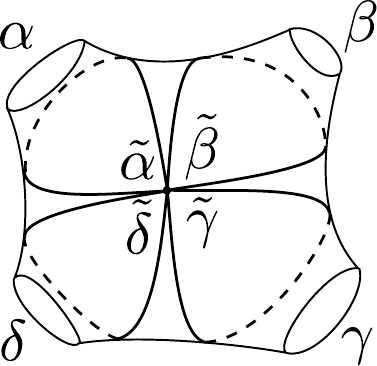}
\caption{Four loops in a four-holed sphere}
\end{center}
\end{figure}
If we assume that $r_p\leq\rho_S$, we get (similarly to before)
$$\tilde{\alpha},\tilde{\beta},\tilde{\gamma}\geq 2\arcsin\left(\sqrt{\frac{3}{7}}\right).$$
Consider the quadrilateral with the four loops as sides; three sides have the same length $2r_p$ and the fourth has length at least $2r_p$. The two diagonals of the quadrilateral are longer than $2r_p$, otherwise we have three loops of length $2r_p$ determining a $3$--holed sphere. Let's denote the angles as in Figure \ref{quadrilateral}.
\begin{figure}[H]
\begin{center}
\includegraphics{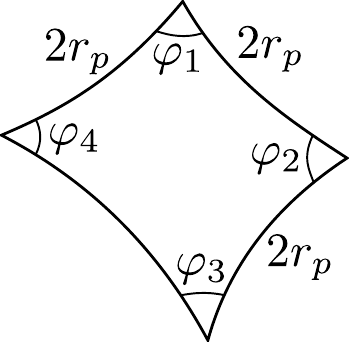}
\caption{The quadrilateral with the four loops as sides}\label{quadrilateral}
\end{center}
\end{figure}
By hyperbolic trigonometry we get
$$\varphi_1,\varphi_2 >2\arcsin\left(\frac{\sqrt{3}}{2\sqrt{7}}\right)$$
and
$$\varphi_3,\varphi_4 >\arcsin\left(\frac{\sqrt{3}}{2\sqrt{7}}\right).$$
So
$$\tilde{\alpha}+\tilde{b}+\tilde{\gamma}+\tilde{\delta}+2\theta+2\varphi>2\pi,$$
a contradiction. Thus $r(S)>\rho_S$.

{\bf Case (b):} assume $r_p\leq \rho_S$. Consider the one-holed torus determined by the geodesic representatives $\alpha$ and $\beta$ of the two loops of length $2r_p$. Denote its boundary curve or cusp by $\gamma$. Since $\alpha,\beta\leq 2\rho_S$, by the Collar Lemma we have $\alpha,\beta\geq 2\arcsinh\left(\frac{\sqrt{3}}{2}\right)$.

Cut along $\alpha$ and denote by $d$ the shortest path between the two copies of $\alpha$. We have (again, by the Collar Lemma) $2\arcsinh\left(\frac{\sqrt{3}}{2}\right)\leq d\leq \beta\leq 2\rho_S$. If $\gamma$ is not a cusp, then
\begin{gather*}
\cosh(\gamma/2)=\sinh^2(\alpha/2)\cosh(d)-\cosh^2(\alpha/2)=\\
=\sinh^2(\alpha/2)(\cosh(d)-1)+1.
\end{gather*}
So $\frac{17}{8}\leq\cosh(\gamma)\leq \frac{41}{9}$ and the width $w(\gamma)$ of the collar around $\gamma$ satisfies 
$$w(\gamma)>\log(5/4)>\log(2\sqrt{3}).$$

\begin{rmk}\label{distcollar}
By hyperbolic trigonometry, if a point $p$ has distance at least $\log(2/\sqrt{3})$ from the collar around a curve $\alpha<2\rho_S$, then the loop based at $p$ and homotopic to $\alpha$ is at least $2\rho_S$.
\end{rmk}

If $\gamma>2\rho_S$, fix $q\in \gamma$. Consider a loop based at $q$ of length $2r_q$ and its geodesic representative $\delta$. There are three possibilities:
\begin{itemize}
\item if $\delta\cap\gamma=\emptyset$, by Remark \ref{distcollar} we get $2r_q=\delta>2\rho_S$;
\item if $\delta\cap\gamma\neq\emptyset$ and $\delta\neq\gamma$, then $\delta$ crosses the one-holed torus, so it crosses $\alpha$ or $\beta$ at least once and $\gamma$ at least twice. Thus $$2r_q=\delta\geq 2\arcsinh(\sqrt{3}/2)+4\log(5/4)>2\rho_S;$$
\item if $\delta=\gamma$, then $r_q>\rho_S$.
\end{itemize}
In all cases, $r_q>\rho_S>r_q$, a contradiction.

Suppose then that $\gamma$ is a curve with $\gamma\leq 2\rho_S$. One can show that in these conditions there exists a solution to the system (\ref{system}), determining a point $q$ with loops of length $\ell>2\arccosh(\frac{\sqrt{37}}{3})>2\rho_S$. Moreover
 $\cosh(d(q,\alpha))=\frac{\sinh(\ell/2)}{\sinh(\alpha/2)}>\rho_S$. So there exists $r>\rho_S$ such that all loops based at $q$ have length at least $2r$, thus again $r(S)>r_p$, a contradiction.
 
If $\gamma$ is a cusp, cut along $\alpha$ and consider a point $q$ which is equidistant from the two copies of $\alpha$ and at distance $\log(2/\sqrt{3})+\varepsilon$ (for $\varepsilon>0$ small) from the horoball of area $2$. By explicit computations, $r_q>\rho_S>r_p$, contradiction again.

So in case (b), $r(S)>\rho_S$.
\end{proof}

\bigskip
\bibliographystyle{plain}
\bibliography{references}
\end{document}